\documentclass[secthm,seceqn,amsthm,ussrhead,reqno, 12pt]{amsart}
\usepackage[utf8]{inputenc}
\usepackage[english]{babel}
\usepackage[symbol]{footmisc}
\usepackage{amssymb,amsmath,amsthm,amsfonts,xcolor,enumerate,hyperref,comment,longtable,cleveref}

\usepackage{times}
\usepackage{cite}
\usepackage{pdflscape}
\usepackage{ulem}
\usepackage[mathcal]{euscript}
\usepackage{tikz}
\usepackage{hyperref}
\usepackage{cancel}
\usepackage{stmaryrd}

\newtheorem{thrm}{Theorem} 

\newtheorem{lem}[thrm]{Lemma}
\newtheorem{prop}[thrm]{Proposition}

\theoremstyle{definition}
\newtheorem{defn}[thrm]{Definition}

\newtheorem{rem}[thrm]{Remark}
\newtheorem*{theoremA}{Theorem A}
\newtheorem*{theoremB}{Theorem B}

\crefrangeformat{equation}{#3(#1)#4--#5(#2)#6}

\crefname{thrm}{Theorem}{Theorems}
\crefname{lem}{Lemma}{Lemmas}
\crefname{cor}{Corollary}{Corollaries}
\crefname{prop}{Proposition}{Propositions}
\crefname{defn}{Definition}{Definitions}
\crefname{exm}{Example}{Examples}
\crefname{rem}{Remark}{Remarks}
\crefname{section}{Section}{Sections}
\crefname{equation}{\unskip}{\unskip}
\crefname{enumi}{\unskip}{\unskip}

\DeclareMathOperator{\spn}{span}

\renewcommand{\iff}{\Leftrightarrow}

\newcommand{\af}{\alpha}
\newcommand{\bt}{\beta}

\newcommand{\Dl}{\Delta}

\newcommand{\lb}{\lambda}

\newcommand{\vf}{\varphi}

\newcommand{\B}{\mathcal{B}}
\newcommand{\C}{\mathbb{C}}
\newcommand{\Z}{\mathbb{Z}}

\newcommand{\LL}{\mathfrak{L}}

\newcommand{\sst}{\subseteq}

\newcommand{\id}{\mathrm{id}}

\newcommand{\ch}{\mathrm{char}}

\newcommand{\pr}[2]{\langle #1,#2 \rangle}
\newcommand{\ot}{\otimes}

\begin{document}

	\noindent{\Large  
		Transposed Poisson structures on \\generalized Witt algebras and Block Lie algebras}\footnote{
		The work was supported by  
		FCT   UIDB/MAT/00212/2020, UIDP/MAT/00212/2020, 2022.02474.PTDC and by CMUP, member of LASI, which is financed by national funds through FCT --- Fundação para a Ciência e a Tecnologia, I.P., under the project with reference UIDB/00144/2020.} 
	\footnote{Corresponding author: kaygorodov.ivan@gmail.com}
	
	\bigskip

	{\bf
		Ivan Kaygorodov\footnote{CMA-UBI, Universidade da Beira Interior, Covilh\~{a}, Portugal; \    kaygorodov.ivan@gmail.com}	\&
		Mykola Khrypchenko\footnote{ Departamento de Matem\'atica, Universidade Federal de Santa Catarina,     Brazil; \and CMUP, Departamento de Matemática, Faculdade de Ciências, Universidade do Porto,
			Rua do Campo Alegre s/n, 4169--007 Porto, Portugal\ nskhripchenko@gmail.com}
	}
	\

	\bigskip
 
	\ 
	
	\noindent {\bf Abstract:} {\it 	
		We describe transposed Poisson structures on generalized Witt algebras $W(A,V, \langle \cdot,\cdot \rangle )$ and Block Lie algebras $L(A,g,f)$ over a field $F$ of characteristic zero, where $\langle \cdot,\cdot \rangle$ and $f$ are non-degenerate. More specifically, if $\dim(V)>1$, then all the transposed Poisson algebra structures on $W(A,V,\langle \cdot,\cdot \rangle)$ are trivial; and if $\dim(V)=1$, then such structures are, up to isomorphism, mutations of the group algebra structure on $FA$. The transposed Poisson algebra structures on $L(A,g,f)$ are in a one-to-one correspondence with commutative and associative multiplications defined on a complement of the square of $L(A,g,f)$ with values in the center of $L(A,g,f)$. In particular, all of them are  usual Poisson structures on $L(A,g,f)$. This generalizes earlier results about transposed Poisson structures on Block Lie algebras $\mathcal{B}(q)$.
	}

	\
	

	

	\section*{Introduction} 
	Poisson algebras originated from the Poisson geometry in the 1970s and have shown their importance in several areas of mathematics and physics, such as Poisson manifolds, algebraic geometry, operads, quantization theory, quantum groups, and classical and quantum mechanics. One of the popular topics in the theory of Poisson algebras is the study of all possible Poisson algebra structures with fixed Lie or associative part~\cite{said2,YYZ07,jawo,kk21}.
	Recently, Bai, Bai, Guo, and Wu~\cite{bai20} have introduced a dual notion of the Poisson algebra, called \textit{transposed Poisson algebra}, by exchanging the roles of the two binary operations in the Leibniz rule defining the Poisson algebra. 
	They have shown that a transposed Poisson algebra defined this way not only shares common properties of a Poisson algebra, including the closedness under tensor products and the Koszul self-duality as an operad, but also admits a rich class of identities. More significantly, a transposed Poisson algebra naturally arises from a Novikov-Poisson algebra by taking the commutator Lie algebra of the Novikov algebra.
	Thanks to \cite{bfk22}, 
	any unital transposed Poisson algebra is
	a particular case of a ``contact bracket'' algebra 
	and a quasi-Poisson algebra.
	Later, in a recent paper by Ferreira, Kaygorodov, and  Lopatkin
	a relation between $\frac{1}{2}$-derivations of Lie algebras and 
	transposed Poisson algebras has been established \cite{FKL}. 	These ideas were used to describe all transposed Poisson structures 
	on  Witt and Virasoro algebras in  \cite{FKL};
	on   twisted Heisenberg-Virasoro,   Schr\"odinger-Virasoro  and  
	extended Schr\"odinger-Virasoro algebras in \cite{yh21};
	on   oscillator algebras in  \cite{bfk22};
        Witt type Lie algebras in \cite{kk23}.
	It was proved that each complex finite-dimensional solvable Lie algebra has a non-trivial transposed Poisson structure \cite{klv22}.
	The ${\rm Hom}$- and ${\rm BiHom}$-versions of transposed Poisson algebras and
	transposed Poisson bialgebras have been considered in \cite{hom, bihom}.	 The algebraic and geometric classification of $3$-dimensional transposed Poisson algebras is given in \cite{bfk23}.	
 For the list of actual open questions on transposed Poisson algebras see \cite{bfk22}.
	
	The first non-trivial example of a transposed Poisson algebra was constructed on the Witt algebra with the multiplication 
	law $[e_i,e_j]=(i-j)e_{i+j}$ for $i,j \in \mathbb  Z$ (see, \cite{FKL}).
	This attracted certain interest to the description of transposed Poisson structures on Lie algebras related to the Witt algebra.
	Thus, all transposed Poisson structures   on 
	the Virasoro algebra  \cite{FKL},
	Block type Lie algebras   and  Block type Lie superalgebras  \cite{kk22},
        Witt type Lie algebras \cite{kk23}
	have been described.
In the last years, the concept of Witt type and Block type Lie algebra has been enlarged and generalized by various authors, such as 
Kawamoto,  Osborn, Đoković, Zhao,  Xu,  Passman,  Jordan, etc. (see, for example, \cite{kk22,kk23} and references therein).
	In the present paper, we study transposed Poisson structures on the class of
	generalized Witt  algebras  defined by 	Đoković and Zhao in \cite{dz96}
 and Block  algebras defined by  Block in~\cite{block58}.
 We use the standard method of characterization of transposed Poisson algebra structures on a fixed Lie algebra $\LL$ based on the description of the space of $\frac{1}{2}$-derivations of $\LL$.
	 
	\medskip 

 Our work consists of two main parts.
	\cref{sec-TP-on-W} is devoted to a description of $\frac{1}{2}$-derivations and transposed Poisson structures on generalized Witt algebras $W(A,V,\pr\cdot\cdot)$, which result in  the following theorem.
 
		\begin{theoremA}[\cref{desc-TP-W}]
		Let $W(A,V,\pr\cdot\cdot)$ be a generalized Witt algebra with non-degenerate $\pr\cdot\cdot$  and $\mathrm{char}(F)=0.$
		\begin{enumerate}
			\item  If $\dim(V)>1$, then all the transposed Poisson algebra structures on $W(A,V,\pr\cdot\cdot)$ are trivial. 
			\item  If $\dim(V)=1$, say, $V=\spn_F\{v\}$, then the transposed Poisson algebra structures on $W(A,V,\pr\cdot\cdot)$ are exactly mutations of the product $(a\ot v)\cdot(b\ot v)=(a+b)\ot v$.
		\end{enumerate}
	\end{theoremA}
	
In \cref{sec-TP-Block} we study the same questions on Block  Lie algebras $L(A,g,f)$ and obtain the following result.

\begin{theoremB} [\cref{descr-TP-on-L-g=0,descr-TP-on-L-g-ne-0}]
Let $L(A,g,f)$ be a Block  Lie algebra with non-degenerate $f$ and $\mathrm{char}(F)=0.$ 
\begin{enumerate}
\item If $g=0,$ then there is only one non-trivial transposed Poisson algebra structure $\cdot$ on $L(A,0,f)$. It is given by $u_0 \cdot u_0=u_0.$

\item If $g\ne 0$
and $(g(a),h(a))\ne (0,-1)$ for all $a\in A$, then all the transposed Poisson algebra structures on $L(A,g,f)$ are trivial.

\item If $g\ne 0$
and there is $a\in A$, such that $(g(a),h(a))\ne (0,-1)$, then the  transposed Poisson algebra structures on $L(A,g,f)$ are usual Poisson algebra structures that are extensions by zero of commutative associative products $*$ on the complement $V=\spn_F\{u_a\mid g(a)=h(a)+2=0\}$ of $[L, L]$ with values in $Z(L)=\spn_F\{u_a\mid g(a)=h(a)+1=0\}$.
\end{enumerate}
\end{theoremB}

	\section{Definitions and preliminaries}\label{prelim}
	
	All the algebras below will be over a field $F$ of characteristic zero and all the linear maps will be $F$-linear, unless otherwise stated.

	\begin{defn}\label{tpa}
		Let ${\mathfrak L}$ be a vector space equipped with two nonzero bilinear operations $\cdot$ and $[\cdot,\cdot].$
		The triple $({\mathfrak L},\cdot,[\cdot,\cdot])$ is called a \textit{transposed Poisson algebra} if $({\mathfrak L},\cdot)$ is a commutative associative algebra and
		$({\mathfrak L},[\cdot,\cdot])$ is a Lie algebra that satisfies the following compatibility condition
		\begin{align}\label{Trans-Leibniz}
			2z\cdot [x,y]=[z\cdot x,y]+[x,z\cdot y].
		\end{align}
	\end{defn}
	
	Transposed Poisson algebras were first introduced in a paper by Bai, Bai, Guo, and Wu \cite{bai20}.
	
	\begin{defn}\label{tp-structures}
		Let $({\mathfrak L},[\cdot,\cdot])$ be a Lie algebra. A \textit{transposed Poisson algebra structure} on $({\mathfrak L},[\cdot,\cdot])$ is a commutative associative multiplication $\cdot$ on $\mathfrak L$ which makes $({\mathfrak L},\cdot,[\cdot,\cdot])$ a transposed Poisson algebra.
	\end{defn}

	\begin{defn}\label{12der}
		Let $({\mathfrak L}, [\cdot,\cdot])$ be an algebra and $\varphi:\mathfrak L\to\mathfrak L$ a linear map.
		Then $\varphi$ is a \textit{$\frac{1}{2}$-derivation} if it satisfies
		\begin{align}\label{vf(xy)=half(vf(x)y+xvf(y))}
			\varphi \big([x,y]\big)= \frac{1}{2} \big([\varphi(x),y]+ [x, \varphi(y)] \big).
		\end{align}
	\end{defn}
	Observe that $\frac{1}{2}$-derivations are a particular case of $\delta$-derivations introduced by Filippov in \cite{fil1}
	(see also \cite{z10} and references therein). The space of all $\frac{1}{2}$-derivations of an algebra $\mathfrak L$ will be denoted by $\Dl(\mathfrak L).$

	\cref{tpa,12der} immediately imply the following key Lemma.
	\begin{lem}\label{glavlem}
		Let $({\mathfrak L},[\cdot,\cdot])$ be a Lie algebra and $\cdot$ a new binary (bilinear) operation on ${\mathfrak L}$. Then $({\mathfrak L},\cdot,[\cdot,\cdot])$ is a transposed Poisson algebra 
		if and only if $\cdot$ is commutative and associative and for every $z\in{\mathfrak L}$ the multiplication by $z$ in $({\mathfrak L},\cdot)$ is a $\frac{1}{2}$-derivation of $({\mathfrak L}, [\cdot,\cdot]).$
	\end{lem}
	
	The basic example of a $\frac{1}{2}$-derivation is the multiplication by a field element.
	Such $\frac{1}{2}$-derivations will be called \textit{trivial}.
	
	\begin{thrm}\label{princth}
		Let ${\mathfrak L}$ be a Lie algebra without non-trivial $\frac{1}{2}$-derivations.
		Then all transposed Poisson algebra structures on ${\mathfrak L}$ are trivial.
	\end{thrm}
	
	Given a Lie algebra $(\LL,[\cdot,\cdot])$ denote by $Z(\LL)$ its \textit{center}, i.e. $Z(\LL)=\{a\in\LL\mid[a,b]=0,\ \forall b\in\LL\}$, and by $[\LL,\LL]$ its square, i.e. $[\LL,\LL]=\spn_F\{[a,b]\mid a,b\in\LL\}$. Fix a complement $V$ of $[\LL,\LL]$ in $\LL$. Then any commutative associative product $*:V\times V\to Z(\LL)$ defines a transposed Poisson algebra structure $\cdot$ on $\LL$ by means of
	\begin{align}\label{(a_1+a_2)-times-(b_1+b_2)}
		(a_1+a_2)\cdot(b_1+b_2)=a_1*b_1,
	\end{align}
	where $a_1,b_1\in V$ and $a_2,b_2\in [\LL,\LL]$. Indeed, the right-hand side of \cref{Trans-Leibniz} is zero, because $z\cdot x,z\cdot y\in Z(\LL)$, and the left-hand side of \cref{Trans-Leibniz} is zero by \cref{(a_1+a_2)-times-(b_1+b_2)}, because $[x,y]\in[\LL,\LL]$. We say that $\cdot$ is \textit{the extension by zero} of $*$. Observe that $\cdot$ is at the same time a usual Poisson structure on $(\LL,[\cdot,\cdot])$.
	

	

	\section{Transposed Poisson structures on generalized Witt algebras}\label{sec-TP-on-W}

\subsection{Generalized Witt algebras}\label{sec-gen-witt}
	
	Đoković and Zhao~\cite{dz96} introduced the following generalization of the classical Witt algebra. 
	\begin{defn}\label{defn-W(A_V_<>)}
		Let $F$ be a field, $(A,+)$ a non-trivial abelian group, $V\ne\{0\}$ a vector space and $\pr\cdot\cdot:V\times A\to F$ a map linear in the first variable and additive in the second one. Denote $W:=FA\otimes_F V$ and define the product $[\cdot,\cdot]$ on $W$ by means of
		\begin{align}\label{[a-tens-b_b-tens-w]=(a+b)-tens-(<v_b>w-<w_a>v)}
			[a\otimes v,b\otimes w]=(a+b)\otimes(\pr vb w-\pr wa v).
		\end{align}
	   Then $(W,[\cdot,\cdot])$ is a Lie algebra called a \textit{generalized Witt algebra}.
	\end{defn}
	When it is necessary to specify $A$, $V$ and $\pr\cdot\cdot$, one writes $W=W(A,V,\pr\cdot\cdot)$. We assume that $\pr\cdot\cdot$ is \textit{non-degenerate}, i.e.
	\begin{align}
		\pr Va=\{0\}&\iff a=0.\label{<a_V>=0-iff-a=0}
	\end{align}
	We also assume that $\mathrm{char}(F)=0$. Then it follows from \cref{<a_V>=0-iff-a=0} that $A$ is torsion-free.
	
	The algebra $W(A,V,\pr\cdot\cdot)$ is a generalization of the so-called \textit{Witt type Lie algebra} $V(f)$ (corresponding to an additive map $f$) introduced\footnote{Notice that `$V$' in $V(f)$ is not the same space $V$ from \cref{defn-W(A_V_<>)}.} by Yu in~\cite{Yu97}. We recall its definition using the notation from the present paper. Given an abelian group $A$, a field $F$ and a function $f:A\to F$, define $V(f)$ to be a vector space with basis $\{e_a\}_{a\in A}$ and multiplication
	\begin{align}\label{[e_a_e_b]=(f(b)-f(a))e_(a+b)}
		[e_a,e_b]=(f(b)-f(a))e_{a+b}.
	\end{align}
Without loss of generality, one assumes that $f(0)=0$. Then $V(f)$ is a Lie algebra if and only if
\begin{align*}
	\big(f(a+b)-f(a)-f(b)\big)\big(f(a)-f(b)\big)=0
\end{align*}
for all $a,b\in A$. Observe that in general one does not require that $f$ be additive. However, it turns out to be so if $|f(A)|\ge 4$ by \cite[Lemma 4.6]{Yu97}.

	\begin{lem}\label{W-cong-V(f)}
		Let $\dim(V)=1$ and  $\pr\cdot\cdot$ be non-degenerate. Then $W(A,V,\pr\cdot\cdot)$ is isomorphic to the Witt type Lie algebra $V(f)$ for some additive injective $f:A\to F$ with $|f(A)|=\infty$.
	\end{lem}
	\begin{proof}
		Choose $v\in V\setminus\{0\}$. Then $\{a\ot v\mid a\in A\}$ is a basis of $W$. Define an additive map $f:A\to F$ by $f(a)=\pr va$ and a bijective linear map $\vf:W\to V(f)$ by $\vf(a\ot v)=e_a$. Then by \cref{[e_a_e_b]=(f(b)-f(a))e_(a+b),[a-tens-b_b-tens-w]=(a+b)-tens-(<v_b>w-<w_a>v)}
		\begin{align*}
			[\vf(a\ot v),\vf(b\ot v)]&=[e_a,e_b]=(f(b)-f(a))e_{a+b}=(\pr vb-\pr va)e_{a+b}\\
			&=(\pr vb-\pr va)\vf((a+b)\ot v)=\vf([a\otimes v,b\otimes v]).
		\end{align*}
	Observe that $f$ is injective by \cref{<a_V>=0-iff-a=0}, because $\pr va=0\iff \pr Va=\{0\}$. Since $A$ is torsion-free, then $|A|=\infty$, whence $|f(A)|=\infty$ as well.
	\end{proof}

Since transposed Poisson structures on $V(f)$ were described in~\cite{kk23}, we only need to deal with the case $\dim(V)>1$.

	\begin{lem}\label{<v'_a>_<v''_a>-ne-0}
		Let $\dim(V)>1$. If $a\ne 0$, then there exist two linearly independent $v',v''\in V$ such that $\pr{v'}a\ne 0\ne \pr{v''}a$.
	\end{lem}
	\begin{proof}
		Denote $V_0=\{v\in V\mid \pr va=0\}$. If $V_0=\{0\}$, there is nothing to prove. Otherwise, choose $0\ne v_0\in V_0$. Since $a\ne 0$, by \cref{<a_V>=0-iff-a=0} there is $v'\in V$ such that $\pr{v'}a\ne 0$. Observe that $v'$ and $v_0$ are linearly independent, since otherwise $v'=kv_0$ and $\pr{v'}a=k\pr{v_0}a=0$. Then $v'$ and $v'+v_0$ are also linearly independent and $\pr{v'+v_0}a=\pr{v'}a\ne 0$. So, we may choose $v''=v'+v_0$.
	\end{proof}

\subsection{$\frac{1}{2}$-derivations of generalized Witt algebras}
 Observe that $W$ is an $A$-graded algebra, namely
	\begin{align*}
		W=\bigoplus_{a\in A} W_a, \text{ where }W_a=a\otimes V=\{a\otimes v\mid v\in V\}.
	\end{align*} 
	For all $a\in A$ and $v\in V$ denote, for simplicity, 
	\begin{align}\label{v_a=a-ot-v}
		v_a:=a\otimes v.
	\end{align}
Any linear map $\vf:W\to W$ decomposes as
	\begin{align*}
		\vf=\sum_{a\in A}\vf_a,
	\end{align*}
	where $\vf_a:W\to W$ is a linear map such that $\vf_a(W_b)\sst W_{a+b}$ for all $b\in A$. In particular, $\vf\in\Dl(W)$ if and only if $\vf_a\in\Dl(W)$ for all $a\in A$. We write 
	\begin{align}\label{dd_a(b-tn-v)=(a+b)-tn-d_a(b-tn-v)}
		\vf_a(v_b)=d_a(v_b)_{a+b},
	\end{align}
	where $d_a:W\to V$.
	
	\begin{lem}\label{conditions-on-d}
		Let $\vf_a:W\to W$ be a linear map satisfying \cref{dd_a(b-tn-v)=(a+b)-tn-d_a(b-tn-v)}. Then $\vf_a\in \Dl (W)$ if and only if for all $x,y\in A$ and $v,w\in V$
		\begin{align} 
&			 2d_a(\pr vy w_{x+y}-\pr wx v_{x+y})=\notag \\
& \quad  \pr {d_a(v_x)}y w-\pr w{a+x} d_a(v_x)
			+\pr v{a+y} d_a(w_y)-\pr {d_a(w_y)}x v.\label{half-der-in-terms-of-d_af}
		\end{align}
	\end{lem}
	\begin{proof}
		By \cref{[a-tens-b_b-tens-w]=(a+b)-tens-(<v_b>w-<w_a>v),dd_a(b-tn-v)=(a+b)-tn-d_a(b-tn-v),v_a=a-ot-v} we have
		\begin{align*}
			2\vf_a([v_x,w_y])&=2\vf_a(\pr vy w_{x+y}-\pr wx v_{x+y})=2d_a(\pr vy w_{x+y}-\pr wx v_{x+y})_{a+x+y}
		\end{align*}
		and
		\begin{align*}
			[\vf_a(v_x),w_y]+[v_x,\vf_a(w_y)]&=[d_a(v_x)_{a+x},w_y]+[v_x,d_a(w_y)_{a+y}]\\
			&=\pr {d_a(v_x)}y w_{a+x+y}-\pr w{a+x} d_a(v_x)_{a+x+y}\\
			&\quad+\pr v{a+y} d_a(w_y)_{a+x+y}-\pr {d_a(w_y)}x v_{a+x+y}.
		\end{align*}
	\end{proof}
	
	\begin{lem}\label{d_a=zero}
		Let  $\dim(V)>1$, $a\ne 0$ and $\vf_a\in\Dl(W)$ satisfying \cref{dd_a(b-tn-v)=(a+b)-tn-d_a(b-tn-v)}. Then $\vf_a=0$.
	\end{lem}
	\begin{proof}
		Substitute $y=0$ into \cref{half-der-in-terms-of-d_af}:
		\begin{align*}
			-2\pr wx d_a(v_x)=-\pr w{a+x} d_a(v_x)+\pr va d_a(w_0)-\pr {d_a(w_0)}x v,
		\end{align*}
	that is
	\begin{align}\label{<w_a-b>d_a(v_b)=<v_a>d_a(w_0)-<d_a(w_0)_b v>}
		\pr w{a-x} d_a(v_x)=\pr va d_a(w_0)-\pr {d_a(w_0)}x v.
	\end{align}
Then setting $x=a$ in \cref{<w_a-b>d_a(v_b)=<v_a>d_a(w_0)-<d_a(w_0)_b v>} we obtain
\begin{align}\label{<v_a>d_a(w_0)=<d_a(w_0)_a>v}
	\pr va d_a(w_0)=\pr {d_a(w_0)}a v.
\end{align}
By \cref{<v'_a>_<v''_a>-ne-0} there are two linearly independent $v',v''\in V$ such that $\pr{v'}a,\pr{v''}a\ne 0$. Choosing consecutively $v=v'$ and $v=v''$ in \cref{<v_a>d_a(w_0)=<d_a(w_0)_a>v} we have
\begin{align*}
	d_a(w_0)=\frac{\pr {d_a(w_0)}a}{\pr {v'}a} v'=\frac{\pr {d_a(w_0)}a}{\pr {v''}a} v''.
\end{align*}
By the linear independence of $v'$ and $v''$, 
\begin{align}\label{d_a(w_0)-zero}
	d_a(w_0)=0
\end{align}
for all $w\in V$. It follows from \cref{<w_a-b>d_a(v_b)=<v_a>d_a(w_0)-<d_a(w_0)_b v>} that
\begin{align*}
	\pr w{a-x} d_a(v_x)=0.
\end{align*}
If $x\ne a$, then $\pr w{a-x}\ne 0$ for some $w\in V$ by \cref{<a_V>=0-iff-a=0}. Thus,
\begin{align}\label{d_a(v_b)=0-for-b-ne-a}
	d_a(v_x)=0,\text{ if }x\ne a.
\end{align}
Now substitute $x=a$ and $y=-a$ into \cref{half-der-in-terms-of-d_af} and use \cref{d_a(w_0)-zero}:
\begin{align*}
0=\pr {d_a(v_a)}a w+2\pr wa d_a(v_a)+\pr {d_a(w_{-a})}a v.
\end{align*}
Since $A$ is torsion-free, then $a\ne -a$, so $d_a(w_{-a})=0$ by \cref{d_a(v_b)=0-for-b-ne-a}.
Taking consecutively $w=v'$ and $w=v''$ we have
\begin{align*}
	d_a(v_a)=-\frac{\pr {d_a(v_a)}a}{2\pr {v'}a} v'=-\frac{\pr {d_a(v_a)}a}{2\pr {v''}a} v'',
\end{align*}
whence 
\begin{align}\label{d_a(v_a)-zero}
	d_a(v_a)=0
\end{align}
by the linear independence of $v'$ and $v''$. Combining \cref{d_a(v_b)=0-for-b-ne-a,d_a(v_a)-zero}, we conclude that $\vf_a=0$.
	\end{proof}
	
\begin{lem}\label{d_gm(af)=d_gm(0)}
	Let $\vf_0\in\Dl(W)$ satisfying \cref{dd_a(b-tn-v)=(a+b)-tn-d_a(b-tn-v)} with $a=0$. Then $\vf_0\in\spn_F\{\id\}$.
\end{lem}
\begin{proof}
	For $a=0$, equality \cref{half-der-in-terms-of-d_af} takes the form
		\begin{align}\label{one-half-der-in-terms-of-d_0}
			2d_0(\pr vy w_{x+y}-\pr wx v_{x+y})=\pr {d_0(v_x)}y w-\pr wx d_0(v_x)
			+\pr vy d_0(w_y)-\pr {d_0(w_y)}x v.
		\end{align}
	Then $y=0$ in \cref{one-half-der-in-terms-of-d_0} gives
	\begin{align*}
		\pr wx d_0(v_x)=\pr {d_0(w_0)}x v.
	\end{align*}
If $x\ne 0$, then choosing $w\in V$ with $\pr wx\ne 0$, we obtain
\begin{align}\label{d_0(v_b)=k_bv}
	d_0(v_x)=\frac{\pr {d_0(w_0)}x}{\pr wx} v=:k_x v,\text{ if }x\ne 0.
\end{align}
In particular, $d_0(v_x)=d_0(v_{-x})$ for all $x\ne 0$.
On the other hand, taking $y=-x\ne 0$ in \cref{one-half-der-in-terms-of-d_0}, we have
\begin{align}\label{d_0-for-c=-b}
	2d_0(\pr vx w_0+\pr wx v_0)=\pr {d_0(v_x)}x w+\pr wx d_0(v_x)+\pr vx d_0(w_x)+\pr {d_0(w_x)}x v,
\end{align}
which for $w=v$ gives
\begin{align*}
	2\pr vx d_0(v_0)=\pr {d_0(v_x)}x v+\pr vx d_0(v_x).
\end{align*}
Choosing $v\in V$ with $\pr vx\ne 0$ and applying \cref{d_0(v_b)=k_bv}, we conclude that
\begin{align}\label{d_0(v_0)=kv-<v_b>-ne-0}
	d_0(v_0)=k_x v=d_0(v_x),\text{ if }\pr vx\ne 0.
\end{align}
If $x\ne 0$ and $\pr vx=0$, then, thanks to \cref{d_0(v_b)=k_bv}, equality \cref{d_0-for-c=-b} becomes 
\begin{align*}
	2\pr wx d_0(v_0)=\pr wx d_0(v_x)+\pr {d_0(w_x)}x v=\pr wx k_xv+\pr {k_xw}x v=2k_x\pr wx v.
\end{align*}
Choosing an arbitrary $w\in V$ with $\pr wx\ne 0$, we arrive at
\begin{align}\label{d_0(v_0)=kv-<v_b>=0}
	d_0(v_0)=k_x v=d_0(v_x),\text{ if }x\ne 0\text{ and }\pr vx=0.
\end{align}
Combining \cref{d_0(v_0)=kv-<v_b>-ne-0,d_0(v_0)=kv-<v_b>=0,dd_a(b-tn-v)=(a+b)-tn-d_a(b-tn-v)}, we finally prove the desired fact.
\end{proof}

\begin{prop}\label{descr-Dl(W)}
	If $\dim(V)>1$, then $\Dl(W)=\spn_F\{\id\}$.
\end{prop}
\begin{proof}
	The inclusion $\Dl(W)\sst\spn_F\{\id\}$ is \cref{d_a=zero,d_gm(af)=d_gm(0)}. The converse inclusion is trivial.
\end{proof}

	\begin{thrm}\label{desc-TP-W}
		Let $\mathrm{char}(F)=0$ and $\pr\cdot\cdot$ be non-degenerate. 
		\begin{enumerate}
			\item\label{TP-W-dim(V)>1} If $\dim(V)>1$, then all the transposed Poisson algebra structures on $W(A,V,\pr\cdot\cdot)$ are trivial. 
			\item\label{TP-W-dim(V)=1} If $\dim(V)=1$, say, $V=\spn_F\{v\}$, then the transposed Poisson algebra structures on $W(A,V,\pr\cdot\cdot)$ are exactly mutations of the product $(a\ot v)\cdot(b\ot v)=(a+b)\ot v$.
		\end{enumerate}
	\end{thrm}
	\begin{proof}
		\cref{TP-W-dim(V)>1} is an immediate consequence of \cref{descr-Dl(W)} and \cite[Theorem 8]{FKL}, while \cref{TP-W-dim(V)=1} follows from \cref{W-cong-V(f)} and \cite[Proposition 26]{kk23}.
	\end{proof}

	\section{Transposed Poisson structures on Block  Lie algebras}\label{sec-TP-Block}	
	\subsection{Block  Lie algebras}
	
	Another generalization of the Witt algebra is the class of Lie algebras studied by Block in~\cite{block58}.
	\begin{defn}
		Let $F$ be a field, $(A,+)$ a non-trivial abelian group, $g:A\to F$ an additive map and $f:A\times A\to F$ an anti-symmetric biadditive map. The \textit{Block algebra} $L(A,g,f)$ is the $F$-vector space with basis $\{u_a\}_{a\in A}$ and product
		\begin{align}\label{[u_a_u_b]=(f(a_b)+g(a-b))u_(a+b)}
			[u_a,u_b]=(f(a,b)+g(a-b))u_{a+b}.
		\end{align}
	\end{defn}
	It is known~\cite{block58} (and, in fact, easy to see) that $L(A,g,f)$ is a Lie algebra if and only if either $g=0$ or there exists an additive map $h:A\to F$ such that for all $a,b\in A$:
	\begin{align}\label{f(a_b)=g(a)h(b)-g(b)h(a)}
		f(a,b)=g(a)h(b)-g(b)h(a).
	\end{align}
	We will write $L=L(A,g,f)$ for the simplicity of notation. We will also assume that $\ch(F)=0$ and $f$ is \textit{non-degenerate} in the sense that
	\begin{align}\label{f(a_A)=0-iff-a=0}
		f(a,A)=\{0\}\iff a=0.
	\end{align}
	Then, as in \cref{sec-gen-witt}, this implies that $A$ is torsion-free. 
	
	Observe that $L(A,g,f)$ is a generalization of the \textit{Block Lie algebra} $\B(q)$ studied in \cite{kk22} (it had been introduced in \cite{xyz12} under slightly different assumptions on $q$ and on the basis). Recall that $\B(q)$, where $q\in\C$, is the complex Lie algebra with a basis $\{L_{m,i} \mid m, i \in \Z\}$, where
	\begin{align*}
		[L_{m,i}, L_{n,j}] = (n(i + q) - m(j + q))L_{m+n,i+j}
	\end{align*}
	for all $i, j, m, n \in \Z$. It is immediately seen that $\B(q)=L(A,g,f)$, where $F=\C$, $A=\Z\times\Z$, 
	\begin{align}\label{g-and-f-for-B(q)}
		g(m,i)=-qm\text{ and }f((m,i),(n,j))=ni-mj
	\end{align}
for $(m,i),(n,j)\in\Z\times\Z$. If $q\ne 0$, then the corresponding map $h$ from \cref{f(a_b)=g(a)h(b)-g(b)h(a)} can be chosen to be
\begin{align}\label{h-for-B(q)}
	h(m,i)=i/q.
\end{align}
Observe that $f$ is non-degenerate, because $f((m,i),(0,-1))=m$ and $f((m,i),(1,0))=i$.
	
	We will need descriptions of $Z(L)$ and $[L,L]$ in the general case.
	\begin{lem}\label{Z(L)-and-[L_L]}
		Let $f$ be non-degenerate.
		\begin{enumerate}
			\item\label{item-Z(L)} If $g=0$, then $Z(L)=\spn_F\{u_0\}$. Otherwise, $Z(L)=\spn_F\{u_a\mid g(a)=h(a)+1=0\}$.
			\item\label{item-[L_L]} If $g=0$, then $[L,L]=\spn_F\{u_a\mid a\ne 0\}$. Otherwise, $[L,L]=\spn_F\{u_a\mid g(a)\ne 0\text{ or }h(a)+2\ne 0\}$.
		\end{enumerate}
	\end{lem}
	\begin{proof}
		\textit{\cref{item-Z(L)}.} Let $g=0$. The inclusion $\spn_F\{u_0\}\sst Z(L)$ is trivial. Conversely, if $x=\sum x_au_a\in Z(L)$ and $x_a\ne 0$ for some $a\ne 0$, then choose $b\in A$ such that $f(a,b)\ne 0$ (it exists due to non-degeneracy of $f$) and calculate 
		\begin{align*}
			[x,u_b]=\sum_{c\ne a} x_c f(c,b)u_{c+b}+x_af(a,b)u_{a+b}\ne 0.
		\end{align*}
		
		Let $g\ne 0$. If $g(a)=h(a)+1=0$, then for all $b\in A$
		\begin{align*}
			f(a,b)+g(a-b)=g(a)h(b)-g(b)h(a)+g(a)-g(b)=g(b)-g(b)=0,
		\end{align*}
		so $[u_a,u_b]=0$. This proves the inclusion $\spn_F\{u_a\mid g(a)=h(a)+1=0\}\sst Z(L)$. Conversely, assume that $x=\sum x_au_a\in Z(L)$. Then $[x,u_0]=0$ implies $g(a)=0$ for all $a$ with $x_a\ne 0$. Consequently, $[x,u_b]=-g(b)\sum x_a(h(a)+1)u_{a+b}$. Choosing $b\in A$ with $g(b)\ne 0$, we conclude that $h(a)+1=0$ whenever $x_a\ne 0$. Thus, $Z(L)\sst\spn_F\{u_a\mid g(a)=h(a)+1=0\}$.
		
		\textit{\cref{item-[L_L]}.} Let $g=0$. If $[u_a,u_b]\ne 0$, then $b\ne -a$, since otherwise $f(a,b)=f(a,-a)=0$. Hence, $[u_a,u_b]=f(a,b)u_{a+b}\in \spn_F\{u_a\mid a\ne 0\}$. This proves $[L,L]\sst\spn_F\{u_a\mid a\ne 0\}$. Conversely, for any $a\ne 0$ there exists $b\in A$ such that $f(a,b)\ne 0$. Then $[u_{a-b},u_b]=f(a-b,b)u_a=f(a,b)u_a\ne 0$, whence $u_a\in[L,L]$.
		
		Let $g\ne 0$. If $[u_a,u_b]\ne 0$, then either $g(a+b)\ne 0$ or $h(a+b)+2\ne 0$, since otherwise
		\begin{align*}
			f(a,b)+g(a-b)&=g(a)h(b)-g(b)h(a)+g(a)-g(b)\\
			&=g(a)(-h(a)-2)+g(a)h(a)+g(a)+g(a)=0.
		\end{align*}
	Hence, $[L,L]\sst\spn_F\{u_a\mid g(a)\ne 0\text{ or }h(a)+2\ne 0\}$. Conversely, take $a\in A$ with $g(a)\ne 0$ or $h(a)+2\ne 0$. If $g(a)\ne 0$, then $[u_a,u_0]=g(a)u_a\ne 0$, so $u_a\in [L,L]$. Otherwise, $h(a)+2\ne 0$ and
	\begin{align*}
		f(a-b,b)+g(a-b-b)&=f(a,b)+g(a-2b)=g(a)h(b)-g(b)h(a)+g(a)-2g(b)\\
		&=-g(b)h(a)-2g(b)=-g(b)(h(a)+2),
	\end{align*}
so choosing $b\in A$ with $g(b)\ne 0$ we have $[u_{a-b},u_b]=-g(b)(h(a)+2)u_a\ne 0$, whence $u_a\in[L,L]$.
	\end{proof}

	We will also need the following technical lemma.
	
	\begin{lem}\label{f(a_c)-ne-0-ne-f(b_c)}
		Let $\af,\bt:A\to F$ two non-zero additive functions. Then there exists $a\in A$ such that $\af(a)\ne 0\ne \bt(a)$.
	\end{lem}
	\begin{proof}
		Assume that for any $a\in A$ either $\af(a)=0$ or $\bt(a)=0$. Then $A=\ker\af\cup\ker\bt$. Since $\ker\af$ and $\ker\bt$ are subgroups of $A$, then either $\ker\af\sst\ker\bt$, in which case $A=\ker\bt$, or $\ker\bt\sst \ker\af$, in which case $A=\ker\af$. Hence, either $\af=0$ or $\bt=0$, a contradiction.
	\end{proof}
 
\subsection{$\frac{1}{2}$-derivations of Block  Lie algebras}
It follows from \cref{[u_a_u_b]=(f(a_b)+g(a-b))u_(a+b)} that $L=\bigoplus_{a\in A} Fu_a$ is an $A$-grading, so any linear map $\vf:L\to L$ decomposes into the direct sum of linear maps
\begin{align*}
	\vf=\sum_{a\in A}\vf_a,
\end{align*}
where $\vf_a(u_b)\in Fu_{a+b}$ for all $b\in A$. Moreover, $\vf\in\Dl(L)$ if and only if $\vf_a\in\Dl(L)$ for all $a\in A$. As usual, we write 
\begin{align}\label{vf_a(u_b)=d_a(u_b)u_(a+b)}
	\vf_a(u_b)=d_a(u_b)u_{a+b},
\end{align}
where $d_a:L\to F$.

\begin{lem}\label{conditions-on-d-case-g=0}
	Let $\vf_a:L\to L$ be a linear map satisfying \cref{vf_a(u_b)=d_a(u_b)u_(a+b)}. Then $\vf_a\in \Dl (L)$ if and only if for all $x,y\in A$
	\begin{align}
		&2(f(x,y)+g(x-y))d_a(x+y)=\notag\\
		&\quad(f(a+x,y)+g(a+x-y))d_a(x)+(f(x,a+y)+g(x-a-y))d_a(y).\label{half-der-in-terms-of-d_a}
	\end{align}
\end{lem}
\begin{proof}
	By \cref{[u_a_u_b]=(f(a_b)+g(a-b))u_(a+b),vf_a(u_b)=d_a(u_b)u_(a+b)} we have
	\begin{align*}
		2\vf_a([u_x,u_y])&=2\vf_a((f(x,y)+g(x-y))u_{x+y})=2(f(x,y)+g(x-y))d_a(x+y) u_{a+x+y}
	\end{align*}
	and
	\begin{align*}
		[\vf_a(u_x),u_y]+[u_x,\vf_a(u_y)]&=[d_a(x)u_{a+x},u_y]+[u_x,d_a(y)u_{a+y}]\\
		&=(f(a+x,y)+g(a+x-y))d_a(x)u_{a+x+y}\\
		&\quad+(f(x,a+y)+g(x-a-y))d_a(y)u_{a+x+y}.
	\end{align*}
\end{proof}
	
	\subsubsection{The case $g=0$}
	
	Assume first that $g=0$.


\begin{lem}\label{d_a=zero-for-a-ne-0}
	Let $a\ne 0$ and $\vf_a\in\Dl(L)$ satisfying \cref{vf_a(u_b)=d_a(u_b)u_(a+b)}. Then $\vf_a=0$.
\end{lem}
\begin{proof}
Taking $y=-x$ in \cref{half-der-in-terms-of-d_a} and using anti-symmetry of $f$, we have
\begin{align*}
	0=f(a+x,-x)d_a(x)+f(x,a-x)d_a(-x)=-f(a,x)(d_a(x)+d_a(-x)).
\end{align*}
Hence,
\begin{align}\label{d_a(-b)=-d_a(b)-for-f(a_b)-ne-0}
	d_a(-x)=-d_a(x),\text{ if }f(a,x)\ne 0.
\end{align}
Now, substitute $y=-a$ into \cref{half-der-in-terms-of-d_a}:
\begin{align*}
	2f(x,-a)d_a(x-a)=f(a+x,-a)d_a(x)=f(x,-a)d_a(x),
\end{align*}
whence 
\begin{align*}
	d_a(x)=2d_a(x-a),\text{ if }f(a,x)\ne 0.
\end{align*}
Since $f(a,a+x)=f(a,x)$, the latter is equivalent to
\begin{align}\label{d_a(a+b)=2d_a(b)-for-f(a_b)-ne-0}
	d_a(a+x)=2d_a(x),\text{ if }f(a,x)\ne 0.
\end{align}
On the other hand, $y=a$ in \cref{half-der-in-terms-of-d_a} gives
\begin{align*}
	2f(x,a)d_a(x+a)=f(a+x,a)d_a(x)+f(x,2a)d_a(a)=f(x,a)d_a(x)+2f(x,a)d_a(a).
\end{align*}
If $f(a,x)\ne 0$, then using \cref{d_a(a+b)=2d_a(b)-for-f(a_b)-ne-0}, we come to $4d_a(x)=d_a(x)+2d_a(a)$. Consequently,
\begin{align}\label{3d_a(b)=2d_a(a)-for-f(a_b)-ne-0}
	3d_a(x)=2d_a(a),\text{ if }f(a,x)\ne 0.
\end{align}
However, $f(a,x)\ne 0\iff f(a,-x)\ne 0$, so replacing $x$ by $-x$ in \cref{3d_a(b)=2d_a(a)-for-f(a_b)-ne-0} and taking into account $\mathrm{char}(F)=0$, we have
\begin{align}\label{d_a(-b)=d_a(b)-for-f(a_b)-ne-0}
	d_a(-x)=d_a(x),\text{ if }f(a,x)\ne 0.
\end{align}
Combining \cref{d_a(-b)=-d_a(b)-for-f(a_b)-ne-0,d_a(-b)=d_a(b)-for-f(a_b)-ne-0}, we conclude that
\begin{align}\label{d_a(b)=0-for-f(a_b)-ne-0}
	d_a(x)=0,\text{ if }f(a,x)\ne 0.
\end{align}

Now assume that $f(a,x)=0$. Since $a\ne 0$, by \cref{f(a_A)=0-iff-a=0} there exists $y\in A$ such that $f(a,y)\ne 0$. Observe that $f(a,x+y)=f(a,y)\ne 0$. Then $d_a(y)=d_a(x+y)=0$ thanks to \cref{d_a(b)=0-for-f(a_b)-ne-0}, so \cref{half-der-in-terms-of-d_a} takes the form
\begin{align*}
	0=f(a+x,y)d_a(x).
\end{align*}
By \cref{f(a_c)-ne-0-ne-f(b_c)} applied to $f(a+x,-)$ and $f(a,-)$, whenever $x\ne -a$, the element $y$ can be chosen in a way that $f(a+x,y)\ne 0\ne f(a,y)$. Thus, we have proved
\begin{align}\label{d_a(b)=0-for-f(a_b)=0-and-b-ne-a}
	d_a(x)=0,\text{ if }f(a,x)=0\text{ and }x\ne -a.
\end{align}

Finally, taking $y=-a-x$ in \cref{half-der-in-terms-of-d_a} we see that the right-hand side is zero, while the left-hand side equals $2f(a,x)d_a(-a)$. Choosing $x\in A$ such that $f(a,x)\ne 0$, we show that $d_a(-a)=0$. Combining this with \cref{d_a(b)=0-for-f(a_b)-ne-0,d_a(b)=0-for-f(a_b)=0-and-b-ne-a}, we get the desired fact.
\end{proof}

\begin{lem}\label{vf_0-in-<id_af>}
	Let $\vf_0\in\Dl(L)$ satisfying \cref{vf_a(u_b)=d_a(u_b)u_(a+b)} with $a=0$. Then $\vf_0(x)=\vf_0(y)$ for all $x,y\ne 0$.
\end{lem}
\begin{proof}
	Write \cref{half-der-in-terms-of-d_a} with $a=0$:
	\begin{align*}
		2f(x,y)d_0(x+y)=f(x,y)d_0(x)+f(x,y)d_0(y).
	\end{align*}
Consequently,
\begin{align}\label{2d_0(b+c)=d_0(b)+d_0(c)-if-f(b_c)-ne-0}
	2d_0(x+y)=d_0(x)+d_0(y),\text{ if }f(x,y)\ne 0.
\end{align}
Observe that $f(x,y)=f(x+y,-y)$, so applying \cref{2d_0(b+c)=d_0(b)+d_0(c)-if-f(b_c)-ne-0} with $(x,y)$ replaced by $(x+y,-y)$, we have
\begin{align}\label{2d_0(b)=d_0(b+c)+d_0(-c)-if-f(b_c)-ne-0}
	2d_0(x)=d_0(x+y)+d_0(-y),\text{ if }f(x,y)\ne 0.
\end{align}
Combining \cref{2d_0(b+c)=d_0(b)+d_0(c)-if-f(b_c)-ne-0,2d_0(b)=d_0(b+c)+d_0(-c)-if-f(b_c)-ne-0}, we come to
\begin{align}\label{3d_0(b)=d_0(c)+2d_0(-c)-if-f(b_c)-ne-0}
	3d_0(x)=d_0(y)+2d_0(-y),\text{ if }f(x,y)\ne 0.
\end{align}
However, $f(x,-y)=-f(x,y)$, so replacing $y$ by $-y$ in \cref{3d_0(b)=d_0(c)+2d_0(-c)-if-f(b_c)-ne-0}, we obtain
\begin{align}\label{3d_0(b)=d_0(-c)+2d_0(c)-if-f(b_c)-ne-0}
	3d_0(x)=d_0(-y)+2d_0(y),\text{ if }f(x,y)\ne 0.
\end{align}
It follows from \cref{3d_0(b)=d_0(c)+2d_0(-c)-if-f(b_c)-ne-0,3d_0(b)=d_0(-c)+2d_0(c)-if-f(b_c)-ne-0} that $d_0(y)=d_0(-y)$, so
\begin{align}\label{d_0(b)=d_0(c)-if-f(b_c)-ne-0}
	d_0(x)=d_0(y),\text{ if }f(x,y)\ne 0,
\end{align}
because $\mathrm{char}(F)=0$.

Now let $x,y\ne 0$. By \cref{f(a_c)-ne-0-ne-f(b_c)} applied to $f(x,-)$ and $f(y,-)$ there exists $z\in A$ such that $f(x,z)\ne 0\ne f(y,z)$. Then \cref{d_0(b)=d_0(c)-if-f(b_c)-ne-0} gives
\begin{align*}
	d_0(x)=d_0(z)=d_0(y),\text{ if }x,y\ne 0,
\end{align*}
as needed.
\end{proof}

\begin{lem}\label{af-in-DL(L)}
	The linear map $\af:L\to L$ given by
	\begin{align*}
		\af(u_a)=
		\begin{cases}
			u_0, & a=0,\\
			0, & a\ne 0,
		\end{cases}
	\end{align*}
is a $\frac 12$-derivation of $L$.
\end{lem}
\begin{proof}
	Observe by \cref{Z(L)-and-[L_L]}\cref{item-Z(L)} that $\af(L)\sst Z(L)$, so the right-hand side of \cref{vf(xy)=half(vf(x)y+xvf(y))} is always zero for $\vf=\af$. Now, $\af([L,L])=\{0\}$ by \cref{Z(L)-and-[L_L]}\cref{item-[L_L]}. Thus, the left-hand side of \cref{vf(xy)=half(vf(x)y+xvf(y))} is always zero as well.
\end{proof}

\begin{prop}\label{descr-Dl(L)}
	We have $\Dl(L)=\spn_F\{\id,\af\}$.
\end{prop}
\begin{proof}
	The inclusion $\Dl(L)\sst\spn_F\{\id,\af\}$ is \cref{d_a=zero-for-a-ne-0,vf_0-in-<id_af>}. The converse inclusion is \cref{af-in-DL(L)}.
\end{proof}

\begin{thrm}\label{descr-TP-on-L-g=0}
Let $\mathrm{char}(F)=0$ and $f$ be non-degenerate. Then there is only one non-trivial transposed Poisson algebra structure $\cdot$ on $L(A,0,f)$. It is given by
\begin{align}\label{u_0-cdot-u_0=u_0}
	u_0\cdot u_0=u_0.
\end{align}
\end{thrm}
\begin{proof}
	Let $\cdot$ be a non-trivial transposed Poisson algebra structure on $L(A,0,f)$. By \cref{descr-Dl(L),glavlem} for any $a\in A$ there are $k_a,l_a\in F$ such that
	\begin{align}\label{u_a-cdot-u_b-general}
		u_a\cdot u_b=k_a u_b+l_a\af(u_b)=
		\begin{cases}
			(k_a+l_a)u_0, & b=0,\\
			k_a u_b, & b\ne 0.
		\end{cases}
	\end{align}
Since $|A|>2$ ($A$ is torsion-free), for any $a\ne 0$ there exists $b\not\in\{0,a\}$. Then by \cref{u_a-cdot-u_b-general} and commutativity of $\cdot$ we have $k_a u_b=u_a\cdot u_b=u_b\cdot u_a=k_b u_a$. Consequently, $k_a=0$ for $a\ne 0$. Similarly, $(k_a+l_a)u_0=u_a\cdot u_0=u_0\cdot u_a=k_0 u_a$ gives $k_0=l_a=0$ for $a\ne 0$. Thus, the only non-zero product $u_a\cdot u_b$ is $u_0\cdot u_0=l_0u_0$. So, up to isomorphism, $\cdot$ is of the form \cref{u_0-cdot-u_0=u_0}.

Conversely, in view of \cref{Z(L)-and-[L_L]} the product \cref{u_0-cdot-u_0=u_0} is of the form \cref{(a_1+a_2)-times-(b_1+b_2)}, so $(L,\cdot,[\cdot,\cdot])$ is a transposed Poisson (and usual Poisson) algebra.
\end{proof}

\begin{rem}
	Consider $\B(0)$ as the complex Block algebra $L(\Z\times\Z,0,f)$, where $f$ is given by \cref{g-and-f-for-B(q)}. Then we obtain the description of transposed Poisson algebra structures on $\B(0)$ given in \cite[Theorem 2.14]{kk22} as a particular case of \cref{descr-TP-on-L-g=0}.
\end{rem}

\subsubsection{The case $g\ne 0$}

In this case, as it was commented above, there exists an additive map $h:A\to F$ such that \cref{f(a_b)=g(a)h(b)-g(b)h(a)} holds.

\begin{lem}\label{d_a-for-g(b)f(a_b)-ne-0}
	Let $a\ne 0$ and $\vf_a\in\Dl(L)$ satisfying \cref{vf_a(u_b)=d_a(u_b)u_(a+b)}. If $g(x)f(a,x)\ne 0$, then $\vf_a(x)=0$.
\end{lem}
\begin{proof}
	Consider first $y=0$ in \cref{half-der-in-terms-of-d_a}
	\begin{align*}
		2g(x)d_a(x)=g(a+x)d_a(x)+(f(x,a)+g(x-a))d_a(0).
	\end{align*}
Then
\begin{align}\label{g(a-b)d_a(b)-is-(f(a_b)+g(a-b))d_a(0)}
	g(a-x)d_a(x)=(f(a,x)+g(a-x))d_a(0).
\end{align}
Replacing $x$ by $-x$, we obtain
\begin{align}\label{g(a+b)d_a(-b)-is-(-f(a_b)+g(a+b))d_a(0)}
	g(a+x)d_a(-x)=(-f(a,x)+g(a+x))d_a(0).
\end{align}
On the other hand, $y=-x$ in \cref{half-der-in-terms-of-d_a} gives
\begin{align}\label{4g(b)d_a(0)-is-(-f(a_b)+g(a+2b))d_a(b)+(-f(a_b)+g(2b-a))d_a(-b)}
	4g(x)d_a(0)=(-f(a,x)+g(a+2x))d_a(x)+(-f(a,x)+g(2x-a))d_a(-x).
\end{align}
Multiplying both sides of \cref{4g(b)d_a(0)-is-(-f(a_b)+g(a+2b))d_a(b)+(-f(a_b)+g(2b-a))d_a(-b)} by $g(a-x)g(a+x)=g(a)^2-g(x)^2$ and using \cref{g(a+b)d_a(-b)-is-(-f(a_b)+g(a+b))d_a(0),g(a-b)d_a(b)-is-(f(a_b)+g(a-b))d_a(0)}, we get
\begin{align}
	4g(x)(g(a)^2-g(x)^2)d_a(0)&=(-f(a,x)+g(a+2x))g(a+x)(f(a,x)+g(a-x))d_a(0)\notag\\
	&\quad+(-f(a,x)+g(2x-a))g(a-x)(-f(a,x)+g(a+x))d_a(0).\label{4g(b)(g(a)^2-g(b)^2)d_a(0)=...}
\end{align}
We have
\begin{align*}
	(-f(a,x)+g(a+2x))(f(a,x)+g(a-x))&=(g(a)+g(x)/2)^2-(f(a,x)-3g(x)/2)^2,\\
	(-f(a,x)+g(2x-a))(-f(a,x)+g(a+x))&=(f(a,x)-3g(x)/2)^2-(g(a)-g(x)/2)^2.
\end{align*}
Since
\begin{align*}
	(g(a)+g(x)/2)^2-(g(a)-g(x)/2)^2&=2g(a)g(x),\\
	(g(a)+g(x)/2)^2+(g(a)-g(x)/2)^2&=2g(a)^2+g(x)^2/2,
\end{align*}
the coefficient of $d_a(0)$ on the right-hand side of \cref{4g(b)(g(a)^2-g(b)^2)d_a(0)=...} equals
\begin{align*}
	&g(a)\cdot 2g(a)g(x)+g(x)(2g(a)^2+g(x)^2/2)-2g(x)(f(a,x)-3g(x)/2)^2\\
	&=g(x)(4g(a)^2+g(x)^2/2-2(f(a,x)-3g(x)/2)^2).
\end{align*}
Subtracting the coefficient of $d_a(0)$ on the left-hand side of \cref{4g(b)(g(a)^2-g(b)^2)d_a(0)=...}, we obtain
\begin{align*}
	g(x)(9g(x)^2/2-2(f(a,x)-3g(x)/2)^2)=2g(x)f(a,x)(3g(x)-f(a,x)).
\end{align*}
Thus, under the assumption $g(x)f(a,x)\ne 0$, \cref{4g(b)(g(a)^2-g(b)^2)d_a(0)=...} is equivalent to
\begin{align}\label{g(b)f(a_b)half-der-in-terms-of-d_a}
	(3g(x)-f(a,x))d_a(0)=0.
\end{align}

\textit{Case 1.} $f(a,x)\ne 3g(x)$. Then \cref{g(b)f(a_b)half-der-in-terms-of-d_a} gives
\begin{align}\label{d_a(0)-is-0}
	d_a(0)=0.
\end{align}

\textit{Case 1.1.} $g(a)\ne g(x)$. It follows from \cref{d_a(0)-is-0,g(a-b)d_a(b)-is-(f(a_b)+g(a-b))d_a(0)} that $d_a(x)=0$.

\textit{Case 1.2.} $g(a)=g(x)$. Then $g(a+x)=2g(x)\ne 0$, so $d_g(-x)=0$ by \cref{g(a+b)d_a(-b)-is-(-f(a_b)+g(a+b))d_a(0)}. Moreover, $-f(a,x)+g(a+2x)=-f(a,x)+3g(x)\ne 0$, so \cref{d_a(0)-is-0,4g(b)d_a(0)-is-(-f(a_b)+g(a+2b))d_a(b)+(-f(a_b)+g(2b-a))d_a(-b)} yield
$d_a(x)=0$.

\textit{Case 2.} $f(a,x)=3g(x)\ne 0$. Then \cref{g(a-b)d_a(b)-is-(f(a_b)+g(a-b))d_a(0)} becomes
\begin{align}\label{g(a-b)d_a(b)-is-(2g(b)+g(a))d_a(0)}
	g(a-x)d_a(x)=(3g(x)+g(a-x))d_a(0)=(2g(x)+g(a))d_a(0).
\end{align}
Since $f(a,x)=3g(x)$ is invariant under the replacement of $x$ by $kx$, then \cref{g(a-b)d_a(b)-is-(2g(b)+g(a))d_a(0)} implies
\begin{align}\label{g(a-kb)d_a(b)-is-(2kg(b)+g(a))d_a(0)}
	g(a-kx)d_a(kx)=(2kg(x)+g(a))d_a(0).
\end{align}
On the other hand, $y=2x$ in \cref{half-der-in-terms-of-d_a} gives
\begin{align}\label{-2g(b)d_a(3b)-is-(5g(b)+g(a))d_a(b)-(4g(b)+g(a))d_a(2b)}
	-2g(x)d_a(3x)=(5g(x)+g(a))d_a(x)-(4g(x)+g(a))d_a(2x).
\end{align}
Multiplying both sides of this equality by $g(a-x)g(a-2x)g(a-3x)$ and using \cref{g(a-kb)d_a(b)-is-(2kg(b)+g(a))d_a(0)} we get
\begin{align}
	&-2g(x)g(a-x)g(a-2x)(6g(x)+g(a))d_a(0)\notag\\
	&=(5g(x)+g(a))g(a-2x)g(a-3x)(2g(x)+g(a))d_a(0)\notag\\
	&\quad-(4g(x)+g(a))g(a-x)g(a-3x)(4g(x)+g(a))d_a(0).\label{-2g(b)g(a-b)g(a-2b)(6g(b)+g(a))d_a(0)=...}
\end{align}
Comparing the coefficients of $g(a)^ig(x)^jd_a(0)$, $0\le i+j\le 4$, in \cref{-2g(b)g(a-b)g(a-2b)(6g(b)+g(a))d_a(0)=...}, we see that \cref{-2g(b)g(a-b)g(a-2b)(6g(b)+g(a))d_a(0)=...} is equivalent to $36g(x)^4d_a(0)=0$. Hence, we again have \cref{d_a(0)-is-0}.

\textit{Case 2.1.} $g(a)\ne g(x)$. Then \cref{d_a(0)-is-0,g(a-b)d_a(b)-is-(f(a_b)+g(a-b))d_a(0)} yield $d_a(x)=0$.

\textit{Case 2.2.} $g(a)=g(x)$. Then $g(a-kx)=(1-k)g(x)\ne 0$ for $k\ne 1$, so $d_a(2x)=d_a(3x)=0$ by \cref{g(a-kb)d_a(b)-is-(2kg(b)+g(a))d_a(0)}. Moreover, $5g(x)+g(a)=6g(x)\ne 0$, so \cref{d_a(0)-is-0,-2g(b)d_a(3b)-is-(5g(b)+g(a))d_a(b)-(4g(b)+g(a))d_a(2b)} imply
$d_a(x)=0$.
\end{proof}

\begin{lem}\label{d_a-for-g(b)f(a_b)=0}
	Let $a\ne 0$ and $\vf_a\in\Dl(L)$ satisfying \cref{vf_a(u_b)=d_a(u_b)u_(a+b)}. If $g(x)f(a,x)=0$, then $d_a(x)=0$, unless $g(a)=g(x)=0$, $h(a)=1$ and $h(x)=-2$.
\end{lem}
\begin{proof}
	\textit{Case 1.} $g(a+x)\ne 0$. Since $a\ne 0$, by \cref{f(a_c)-ne-0-ne-f(b_c)} applied to $f(a,-)$ and $g$ there exists $y\in A$ such that $f(a,y)\ne 0\ne g(y)$. Observe that for any $k\ne 0$ we have
	\begin{align}\label{f(a_c)-ne-0-ne-g(c)}
		f(a,ky)\ne 0\ne g(ky),
	\end{align}
	because $\ch(F)=0$. We affirm that $k$ can be chosen in a way that
	\begin{align}\label{f(a_b+kc)-ne-0-ne-f(a_kc)+g(a)+g(kc)}
		f(a,x+ky)\ne 0\ne f(a+x,ky)+g(a+x-ky).
	\end{align}
	Indeed, $f(a,y)\ne 0$, so there exists at most one $k$ such that $f(a,x)+kf(a,y)=0$. If $f(a+x,y)-g(y)=0$, then $f(a+x,ky)+g(a+x-ky)=g(a+x)\ne 0$ for all $k$. Otherwise, there exists at most one $k$ such that $k(f(a+x,y)-g(y))+g(a+x)=0$. Thus, \cref{f(a_b+kc)-ne-0-ne-f(a_kc)+g(a)+g(kc),f(a_c)-ne-0-ne-g(c)} hold for infinitely many integer $k\ne 0$ (recall that $\ch(F)=0$). Then $d_a(ky)=d_a(x+ky)=0$ for any such $k$ by \cref{d_a-for-g(b)f(a_b)-ne-0}. Consequently, applying \cref{half-der-in-terms-of-d_a} with $y$ replaced by $ky$ and using \cref{f(a_b+kc)-ne-0-ne-f(a_kc)+g(a)+g(kc)} we prove $d_a(x)=0$.
	
	\textit{Case 2.} $g(a+x)=0$. Then 
	\begin{align*}
		f(a+x,y)+g(a+x-y)=-g(y)(h(a+x)+1).
	\end{align*}

	\textit{Case 2.1.} $h(a+x)+1\ne 0$. Then the same argument as in Case 1 gives $d_a(x)=0$. 
	
	\textit{Case 2.2.} $h(a+x)+1=0$. Then $0=g(x)f(a,x)=g(a)^2=g(x)^2$, so \cref{half-der-in-terms-of-d_a} simplifies to
	\begin{align*}
		2g(y)h(a)d_a(x+y)=g(y)h(a)d_a(y).
	\end{align*}
Observe that $h(a)\ne 0$, since otherwise $f(a,y)=0$ for all $y\in A$ contradicting \cref{f(a_A)=0-iff-a=0}. Hence,
\begin{align}\label{2d_a(b+c)-is-d_a(c)}
	2d_a(x+y)=d_a(y),\text{ if }g(y)\ne 0.
\end{align}
On the other hand, since $g(a)=0$, then substituting $x=0$ in \cref{half-der-in-terms-of-d_a} we get
\begin{align}\label{d_a(c)-is-(h(a)+1)d_a(0)}
	d_a(y)=(h(a)+1)d_a(0),\text{ if }g(y)\ne 0.
\end{align}
Observe that $g(x+y)=g(y)$ whenever $g(x)=0$, so \cref{d_a(c)-is-(h(a)+1)d_a(0),2d_a(b+c)-is-d_a(c)} yield
\begin{align*}
	2(h(a)+1)d_a(0)=(h(a)+1)d_a(0).
\end{align*}
It follows that $d_a(0)=0$, unless $h(a)+1=0$. In any case, \cref{d_a(c)-is-(h(a)+1)d_a(0)} implies
\begin{align}\label{d_a(c)-is-zero}
	d_a(y)=0,\text{ if }g(y)\ne 0.
\end{align}
Write \cref{half-der-in-terms-of-d_a} replacing $x$ by $x-y$, where $g(y)\ne 0$. Since $d_a(y)=d_a(x-y)=0$ by \cref{d_a(c)-is-zero}, we come to
\begin{align*}
	(h(x)+2)d_a(x)=0.
\end{align*}
Thus, $d_a(x)=0$, unless $h(x)=-2$ (in which case $h(a)=1$).
\end{proof}

\begin{lem}\label{vf_0-in-<id>}
	Let $\vf_0\in\Dl(L)$ satisfying \cref{vf_a(u_b)=d_a(u_b)u_(a+b)} with $a=0$. Then $\vf_0\in\spn_F\{\id\}$.
\end{lem}
\begin{proof}
	Writing \cref{half-der-in-terms-of-d_a} with $a=0$, we get
	\begin{align}\label{2(f(x_y)+g(x-y))d_0(x+y)-is-(f(x_y)+g(x-y))(d_0(x)+d_0(y))}
		2(f(x,y)+g(x-y))d_0(x+y)=(f(x,y)+g(x-y))(d_0(x)+d_0(y)).
	\end{align}

\textit{Case 1.} $g(x)\ne 0$. Then put $y=0$ in \cref{2(f(x_y)+g(x-y))d_0(x+y)-is-(f(x_y)+g(x-y))(d_0(x)+d_0(y))} to get:
\begin{align}\label{d_0(x)-is-d_0(0)}
	d_0(x)=d_0(0),\text{ if }g(x)\ne 0.
\end{align}

\textit{Case 2.} $g(x)=0$ and $h(x)\ne -1$. Then $f(x,y)=-g(y)h(x)$, and \cref{2(f(x_y)+g(x-y))d_0(x+y)-is-(f(x_y)+g(x-y))(d_0(x)+d_0(y))} is equivalent to
\begin{align}\label{2g(y)d_0(x+y)-is-g(y)(d_0(x)+d_0(y))}
	2g(y)d_0(x+y)=g(y)(d_0(x)+d_0(y)).
\end{align}
Choose $y\in A$ such that $g(y)\ne 0$. Then $d_0(y)=d_0(x+y)=d_0(0)$ by \cref{d_0(x)-is-d_0(0)}. Hence, \cref{2g(y)d_0(x+y)-is-g(y)(d_0(x)+d_0(y))} yields $d_0(x)=d_0(0)$.

\textit{Case 3.} $g(x)=0$ and $h(x)=-1$. Let us use \cref{2(f(x_y)+g(x-y))d_0(x+y)-is-(f(x_y)+g(x-y))(d_0(x)+d_0(y))} with $(x,y)$ replaced by $(x+y,-y)$. Observe that $f(x+y,-y)=-f(x,y)=-g(y)$, so we get
\begin{align}\label{6g(y)d_0(x)-is-3g(y)(d_0(x+y)+d_0(-y))}
	6g(y)d_0(x)=3g(y)(d_0(x+y)+d_0(-y)).
\end{align}
Choosing $y\in A$ with $g(y)\ne 0$ we have $d_0(x+y)=d_0(-y)=d_0(0)$ by \cref{d_0(x)-is-d_0(0)}. Thus, $d_0(x)=d_0(0)$ due to \cref{6g(y)d_0(x)-is-3g(y)(d_0(x+y)+d_0(-y))}.
\end{proof}

Given $\lb,\mu\in F$, we introduce the following notation: 
\begin{align*}
	A_{(\lb,\mu)}:=\{a\in A\mid g(a)=\lb\text{ and }h(a)=\mu\}.
\end{align*}
\begin{lem}\label{af_(a_b)-in-D(L)}
	Let $a\in A_{(0,-2)}$ and $b\in A_{(0,-1)}$. Then the linear map $\af_{(a,b)}:L\to L$ given by
	\begin{align*}
		\af_{(a,b)}(u_c)=
		\begin{cases}
			u_b, & c=a,\\
			0, & \text{otherwise},
		\end{cases}
	\end{align*}
	is a $\frac 12$-derivation of $L$.
\end{lem}
\begin{proof}
		Observe by \cref{Z(L)-and-[L_L]}\cref{item-Z(L)} that $\af_{(a,b)}(L)\sst Z(L)$, showing that the right-hand side of \cref{vf(xy)=half(vf(x)y+xvf(y))} is always zero for $\vf=\af_{(a,b)}$. Furthermore, $\af_{(a,b)}([L,L])=\{0\}$ by \cref{Z(L)-and-[L_L]}\cref{item-[L_L]}, and the left-hand side of \cref{vf(xy)=half(vf(x)y+xvf(y))} is always zero as well.
\end{proof}

\begin{prop}\label{descr-Dl(L)-for-g-ne-0}
	We have $\Dl(L)=\spn_F(\{\id\}\cup\{\af_{(a,b)}\mid a\in A_{(0,-2)}\text{ and }b\in A_{(0,-1)}\})$.
\end{prop}
\begin{proof}
	The fact that any $\vf\in\Dl(L)$ is a linear combination of $\id$ and $\af_{(a,b)}$ follows from \cref{d_a-for-g(b)f(a_b)-ne-0,d_a-for-g(b)f(a_b)=0,vf_0-in-<id>}. Conversely, the inclusion $\id\in\Dl(L)$ is trivial, and $\af_{(a,b)}\in\Dl(L)$ is \cref{af_(a_b)-in-D(L)}.
\end{proof}

\begin{thrm}\label{descr-TP-on-L-g-ne-0}
	Let $\mathrm{char}(F)=0$, $g\ne 0$ and $f$ be non-degenerate. If $(g(a),h(a))\ne (0,-1)$ for all $a\in A$, then all the transposed Poisson algebra structures on $L(A,g,f)$ are trivial. Otherwise, the transposed Poisson algebra structures on $L(A,g,f)$ are exactly extensions by zero of commutative associative products $*$ on the complement $V=\spn_F\{u_a\mid g(a)=h(a)+2=0\}$ of $[L,L]$ with values in $Z(L)=\spn_F\{u_a\mid g(a)=h(a)+1=0\}$.
\end{thrm}
\begin{proof}
	Let $\cdot$ be a transposed Poisson algebra structure on $L(A,g,f)$. If $(g(a),h(a))\ne (0,-1)$ for all $a\in A$, then $A_{(0,-1)}=\emptyset$, so by \cref{descr-Dl(L)-for-g-ne-0} we have $\Dl(L)=\{\id\}$. It follows from \cite[Theorem 8]{FKL} that $\cdot$ is trivial.
	
	Assume that $(g(a),h(a))=(0,-1)$ for some $a\in A$. Then $(g(2a),h(2a))=(0,-2)$, so both $A_{(0,-1)}$ and $A_{(0,-2)}$ are non-empty. By \cref{descr-Dl(L)-for-g-ne-0,glavlem} for any $a\in A$ there are $k_a\in F$ and $\{l_a^{(x,y)}\}_{x\in A_{(0,-2)},y\in A_{(0,-1)}}\sst F$, such that
	\begin{align}\label{u_a-cdot-u_b-general-g-ne-0}
		u_a\cdot u_b=k_au_b+\sum_{x\in A_{(0,-2)},y\in A_{(0,-1)}} l_a^{(x,y)}\af_{(x,y)}(u_b)=
		\begin{cases}
			k_au_b+\sum_{y\in A_{(0,-1)}}l_a^{(b,y)}u_y, & b\in A_{(0,-2)},\\
			k_a u_b, & b\not\in A_{(0,-2)}.
		\end{cases}
	\end{align}
	Since $\ch(F)=0$ and $A$ is torsion-free, then $A\setminus A_{(0,-1)}$ and $A\setminus A_{(0,-2)}$ are infinite (indeed, if $a\in A_{(\lb,\mu)}$ with $(\lb,\mu)\ne (0,0)$, then $ka\not\in A_{(\lb,\mu)}$ for all $k\ne 1$). So, for any $a\not\in A_{(0,-2)}$ there exists $b\not\in A_{(0,-2)}$, $b\ne a$. Then by \cref{u_a-cdot-u_b-general-g-ne-0} and commutativity of $\cdot$ we have $k_a u_b=u_a\cdot u_b=u_b\cdot u_a=k_b u_a$. Consequently, 
	\begin{align}\label{k_a=0-for-a-not-in-A_0_-2}
		k_a=0\text{ for }a\not\in A_{(0,-2)}.
	\end{align} 
Now let $a\in A_{(0,-2)}$ and $b\not\in A_{(0,-1)}\cup A_{(0,-2)}$. Then $k_bu_a+\sum_{y\in A_{(0,-1)}}l_b^{(a,y)}u_y=u_b\cdot u_a=u_a\cdot u_b=k_a u_b$ gives 
\begin{align}\label{k_a=0-for-a-in-A_0_-2}
	k_a=0\text{ for }a\in A_{(0,-2)}.
\end{align}
It follows from \cref{k_a=0-for-a-not-in-A_0_-2,k_a=0-for-a-in-A_0_-2} and the commutativity of $\cdot$ that
\begin{align*}
	u_a\cdot u_b=
	\begin{cases}
		\sum_{y\in A_{(0,-1)}}l_a^{(b,y)}u_y, & a,b\in A_{(0,-2)},\\
		0, & \text{otherwise}.
	\end{cases}
\end{align*}
Thus, $\cdot$ is of the form \cref{(a_1+a_2)-times-(b_1+b_2)} for the commutative associative product $u_a*u_b=\sum_{y\in A_{(0,-1)}}l_a^{(b,y)}u_y\in Z(L)$ on $V=\spn_F\{u_a\mid a\in A_{(0,-2)}\}$.
	
	Conversely, in view of \cref{Z(L)-and-[L_L]} the product \cref{u_0-cdot-u_0=u_0} is of the form \cref{(a_1+a_2)-times-(b_1+b_2)}, so $(L,\cdot,[\cdot,\cdot])$ is a transposed Poisson (and usual Poisson) algebra.
\end{proof}

\begin{rem}
	Consider $\B(q)$ with $q\ne 0$ as the complex Block algebra $L(\Z\times\Z,g,f)$, where $g$, $f$ and $h$ are given by \cref{g-and-f-for-B(q),h-for-B(q)}. Then $(g(m,i),h(m,i))=(0,-2)\iff (m,i)=(0,-2q)$ and $(g(m,i),h(m,i))=(0,-1)\iff (m,i)=(0,-q)$, so we again obtain the description of transposed Poisson algebra structures on $\B(q)$ given in \cite[Theorem 2.14]{kk22} as a particular case of 
 \cref{descr-TP-on-L-g-ne-0}.
\end{rem}


	
	




\end{document}